\documentclass[11pt, twoside, a4paper]{amsart}

\usepackage[utf8]{inputenc}
\usepackage{dirtytalk}
\usepackage{array,multirow}
\usepackage[dvipsnames]{xcolor}
\usepackage{amsfonts, amstext, amsmath, amsthm, amscd, amssymb}
\usepackage{mathrsfs}
\usepackage{tikz-cd}
\usetikzlibrary{cd}
\tikzset{
  symbol/.style={
    draw=none,
    every to/.append style={
      edge node={node [sloped, allow upside down, auto=false]{$#1$}}
    },
  },
}
\usepackage{url}
\usepackage[all]{xypic}
\usepackage{graphicx,import}
\graphicspath{{/Immagini/}}
\usepackage[margin=1.3in]{geometry}
\usepackage{bm}
\usepackage[english]{babel}
\usepackage{enumerate}
\renewcommand{\setminus}{{\smallsetminus}}
\numberwithin{equation}{section}
\theoremstyle{plain}
\newtheorem{theorem}[equation]{Theorem}
\newtheorem{lemma}[equation]{Lemma}

\newtheorem{corollary}[equation]{Corollary}
\newtheorem*{question}{Question}

\newtheorem*{cor:m-doesnotdivide-n}{Corollary~\ref{cor:m-doesnotdivide-n}}
\newtheorem*{cor:examples}{Corollary~\ref{cor:examples}}
\newtheorem*{new}{Theorem~\ref{new}}
\newtheorem*{new1}{Theorem~\ref{new1}}
\newtheorem*{theorem1}{Lemma~\ref{theorem1}}

\theoremstyle{definition}
\newtheorem{example}[equation]{Example}

\newtheorem{definition}[equation]{Definition}

\newtheorem{remark}[equation]{Remark}

\theoremstyle{remark}
\newtheorem*{claim}{Claim}
\newtheorem{const}[equation]{Construction}

\numberwithin{equation}{section}

\usepackage[backend=biber,
style=alphabetic,citestyle=alphabetic,sorting=anyt]{biblatex}
\bibliography{bibliography}
\usepackage{csquotes}
\usepackage{url}
\usepackage[hidelinks]{hyperref}
\hypersetup{
			colorlinks=true,
			linkcolor=red!50!black,
			anchorcolor=blue,
			citecolor=black,
			urlcolor=black
			}
\title{Diffeotopy groups of non-compact 4-manifolds}
\author[Isacco Nonino]{Isacco Nonino}
\address{University of Glasgow, Glasgow, UK}
\email{Isacco.Nonino@glasgow.ac.uk}

\begin{document}

\maketitle
\begin{abstract}
We provide information on diffeotopy groups of exotic smoothings of punctured 4-manifolds, extending previous results on  diffeotopy groups of exotic $\mathbb{R}^4$'s. In particular, we prove that for a smoothable 4-manifold $M$ and for a non-empty, discrete set of points $S \subsetneq \mathring{M}$, there are uncountably many distinct smoothings of $M\setminus S$ whose diffeotopy groups are uncountable. 

We then prove that for a smoothable 4-manifold $M$ and for a non-empty, discrete set of points $S \subsetneq \mathring{M}$, there exists a smoothing of $M\setminus S$ whose diffeotopy groups have similar properties as $\mathcal{R}_U$, Freedman and Taylor's universal $\mathbb{R}^4$.

Moreover, we prove that if $M$ is non-smoothable, both results still hold under the assumption that $|S| \ge 2$.
\end{abstract}
\section{Introduction}
One of the most striking results in 4-dimensional topology is the existence of \textit{uncountably many} different smooth structures on $\mathbb{R}^4$ (see for example \cite{Taubes1987GaugeTO,anexoticmanagerie,gompf-stip}). This wild behavior was later also detected in several other classes of non-compact topological 4-manifolds (see for example \cite{bizaka-entnyre,gompf-stip}). Up to now, it is still an open question whether \emph{all} non-compact topological 4-manifolds admit uncountably many distinct smoothings. In particular, punctured 4-manifolds (i.e, manifolds with points removed)  behave as wildly as $\mathbb{R}^4$. This is not surprising: one way to view $\mathbb{R}^4$ is as the 4-sphere without a point. 

Although during the years a lot was proved about the exotic structures on non-compact 4-manifolds, little was known until recently on how these structures behaved in terms of self-diffeomorphism groups.

Diffeomorphism groups are normally hard to compute. A useful approximation is given by the \emph{mapping class groups}. They consist of diffeomorphisms of a manifold $M$ to itself modulo isotopy. In this paper we will refer to them as \emph{diffeotopy groups}, following the terminology used in \cite{groupactions}. Given a smooth manifold $M$, we denote the diffeotopy group by $\mathcal{D}(M)$. In other words, elements of $\mathcal{D}(M)$ are smooth isotopy classes of self-diffeomorphisms of $M$. 
If $M$ is non-compact, we will also consider \emph{self-diffeomorphisms at infinity}. Roughly speaking they are defined in terms of smooth, proper embeddings of codimension-0 submanifolds whose complement in $M$ has compact closure. They are \emph{isotopic} if they are properly isotopic as proper embeddings. We denote the \emph{diffeotopy group at infinity} of $M$ by $\mathcal{D}^{\infty}(M)$. 

Each self-diffeomorphim $f$ of $M$ defines a diffeomorphism at infinity simply by restriction. We obtain a homomorphism $r\colon \mathcal{D}(M) \to \mathcal{D}^{\infty}(M)$ by sending diffeomorphisms to their corresponding equivalence classes at infinity. If a diffeomorphism at infinity lies in the image of $r$ we say that it \emph{extends} over $M$.

Even if mapping class groups are simpler to analyze than diffeomorphism groups, they still are quite obscure objects.

As an example, we do not know how to compute $\mathcal{D}^{\infty}(\mathbb{R}^4)$ yet (this is in fact related, as explained in \cite{groupactions}, to the \emph{smooth Schoenflies conjecture} in dimension 4, which is notoriously still open).

In a recently published paper, Gompf approached the question, proving that many (in fact, \textit{uncountably} many) smoothings of $\mathbb{R}^4$ have an \textit{uncountable} diffeotopy group $\mathcal{D}(\mathcal{R})$ \cite[Theorem 4.4]{groupactions}. The same holds for their diffeotopy group at infinity $\mathcal{D}^{\infty}(\mathcal{R})$. 

He also proved an interesting result for $\mathcal{R}_U$, Freedman and Taylor's universal smoothing of $\mathbb{R}^4$ \cite{universal}. Though nothing is known about $\mathcal{D}(\mathcal{R}_U)$ and $\mathcal{D}^{\infty}(\mathcal{R}_U)$ by themselves, we now know that the group homomorphism $r\colon \mathcal{D}(\mathcal{R}_U) \to \mathcal{D}^{\infty}(\mathcal{R}_U)$ is surjective \cite[Theorem 5.1]{anexoticmanagerie}. Equivalently, every diffeomorphism at infinity of $\mathcal{R}_U$ extends over $\mathcal{R}_U$. We will refer to this extension property as \emph{universal behavior}.\footnote{ Some care is required: the precise meaning of \emph{universal behavior} depends on the complexity at infinity of the non-compact 4-manifold we are working with. We will discuss this in Section 4.}

As the early results on exotic  $\mathbb{R}^4$'s were used to distinguish between smooth structures on other non-compact 4-manifolds, we will use Gompf's work as the starting point for a general discussion on diffeotopy groups of non-compact 4-manifolds. This possibility is mentioned in \cite[Section 4]{groupactions} and carried out in \cite[Theorem 4.10]{groupactions} for the interior of manifolds having a specific embedding property. However in this paper we follow a different strategy.

We will focus our attention on punctured 4-manifolds.
In Section 3 and Section 4 we will prove the following two theorems:
\begin{new}
Let $M$ be a topological 4-manifold and $ S \subsetneq \mathring{M}$ a non-empty discrete set of points. Then $M\setminus S$ admits uncountably many smoothings whose diffeotopy groups $\mathcal{D}^{\infty}(M\setminus S)$ and $\mathcal{D}(M\setminus S)$ are uncountable if:
\begin{enumerate}
    \item $M$ is smoothable or
    \item $M$ is non-smoothable and $|S| \ge 2$.
\end{enumerate}
\end{new}
\begin{new1}
Let $M$ be a topological 4-manifold, $S \subsetneq \mathring{M}$ a non-empty discrete set of points. 
\begin{enumerate}
    \item If $M$ is smoothable, then $M \setminus S$ admits a smoothing for which the map $r_{\epsilon_p}$ is surjective for each $p \in S$.
    \item If $M$ is non-smoothable and $|S| \ge 2$, then $M \setminus S$ admits a smoothing for which $r_{\epsilon_p}$ is surjective for all but one $p \in S$.
\end{enumerate}
\end{new1}
Here $\mathring{M}$ indicates the \emph{manifold} interior, i.e $M \setminus \partial M$. 

Note that in both Theorem \ref{new} and Theorem \ref{new1} we did not address the case $M$ a topological \emph{non-smoothable} 4-manifold with only one puncture. 
The following natural question arises:
\begin{question}
Do Theorems \ref{new} and \ref{new1} apply also for $M\setminus p$, where $M$ is a topological \emph{non-smoothable}  4-manifold and $p \in \mathring{M}$?
\end{question}

The proofs of the two theorems were  inspired by \cite{furuta} and \cite{anexoticmanagerie}. In the referenced papers uncountability of smooth structures was extended from $\mathbb{R}^4$ to punctured 4-manifolds exploiting the homeomorphism $M\setminus p \approx M \# \mathbb{R}^4$ for $p \in \mathring{M}$. We will do the same here. This is rather intuitive: in \cite[Theorem 4.4]{groupactions}, the uncountability of $\mathcal{D}(\mathcal{R})$ is proven by exhibiting a group injection  $G \hookrightarrow \mathcal{D}^{\infty}(\mathcal{R})$, where $G$ is an uncountable group with the discrete topology. The behavior at infinity was enough to detect non-isotopic diffeomorphisms of $\mathcal{R}$. It is quite natural then to try puncturing a 4-manifold $M$ in such a way to have an end diffeomorphic to that of a previously studied  $\mathbb{R}^4$ homeomorph. We will then need to show how the exotic behavior at infinity can distinguish isotopy classes of self-diffeomorphisms of the punctured manifold.

Our starting point will be the following lemma. This is an easy case of Theorem \ref{new}.
\begin{theorem1} 
Let $M$ be a closed, smooth 4-manifold and $p \in M$. Then there exists a smoothing  of $M\setminus p$ whose diffeotopy groups  $\mathcal{D}^{\infty}(M\setminus p)$ and $\mathcal{D}(M\setminus p)$ are uncountable.
\end{theorem1}
The proof of this result will showcase the techniques used throughout the paper. The main idea is the following: take an uncountable group $G \in \mathcal{G}^*$ \cite{groupactions}. $\mathcal{G}^*$ is the collection of all groups with discrete topology that have an action via diffeomorphisms on $\mathbb{R}^4$ that respects certain properties. Intuitively, the diffeomorphisms defined by elements of $G \in \mathcal{G}^*$ interact equivariantly with a collection of rays properly embedded in $\mathbb{R}^4$ and keep fixed pointwise a region disjoint from these rays. The precise definition of $\mathcal{G}^*$ is given in Section 2.
We construct a $G$-action via diffeomorphisms on the punctured manifold $M \setminus p$. To do so, we will use $M \# \mathbb{R}^4 \approx  M \setminus p$. Once the action is constructed, we exploit Gompf's result to derive a contradiction. We will so prove that $G$ injects into $\mathcal{D}^{\infty}(M \setminus p)$ as well as $\mathcal{D}(M \setminus p)$. 

It is worthwhile to emphasize that diffeotopy groups of \emph{compact} 4-manifolds are mostly shrouded in mystery as well. This showcases once again the discrepancies between dimension 4 and other dimensions. The case of $\mathbb{S}^4$ is for example far from understood. There are however recent breakthroughs by Watanabe, Gabai and Budney-Gabai \cite{watanabe2018some,budney2019knotted,gabai20204} that sparked new life in this area of research, which still has a lot to offer in terms of open questions. 

\subsection{Outline}

 In Section 2 we will recall all the definitions and results we need from Gompf's paper.
Section 3 and 4 will be devoted to the proofs of Theorem \ref{new} and Theorem \ref{new1} respectively.
\subsection{Acknowledgements}
The work presented in this paper is part of my master's thesis \emph{Smooth structures on non-compact 4-manifolds} at the University of Bonn, written under the supervision of Dr. Arunima Ray. I thank her for being extremely supportive and for all the useful discussions we had. I also thank Prof. Robert E. Gompf, whose outstanding work inspired my research and whose kind answers helped me a lot during the last months. 
 
\section{Group actions and Exotic $\mathbb{R}^4$'s}
 We recall the main definitions and results from \cite{groupactions}.
\subsection{Diffeotopy groups at infinity and diffeotopy groups of the ends}

Given a smooth manifold $V$, we study its diffeomorphism group modulo isotopy.
\begin{definition}
Let $V$ be a smooth manifold. The group $\mathcal{D}(V)$ of isotopy classes of self-diffeomorphisms of $V$ is called the \emph{diffeotopy group of $V$}.
\end{definition}

For a non-compact manifold $V$, we also have an analogous notion in terms of diffeomorphisms at infinity. We now give a precise definition of such objects.

\begin{definition}
Given a non-compact manifold $V$, a \emph{closed neighborhood of infinity} is a codimension-0 submanifold $U \subseteq V$ that is a closed subset whose complement has compact closure (see Figure \ref{endneighborhood}).
\end{definition}
\begin{definition}\label{end}
A \emph{compact exhaustion} of a connected topological manifold $M$ is a sequence $\{C_i\}_{i=1}^{\infty}$ of compact, connected, codimension-0 submanifolds with $C_i \subseteq Int_M(C_{i+1}) $ (topological interior) and $\bigcup_i C_i= M$. By letting $N_i:= \overline{M \setminus C_i}$, we obtain a corresponding \emph{cofinal sequence of closed neighborhoods of infinity}. Each such $N_i$ has connected components $\{N_i^j\}_{j=1}^{k_i}$, where $k_i= \infty$ is also allowed. We can arrange that each $N_i^j$ is noncompact by enlarging $C_i$ to include all the compact components of $N_i$.

In the above arrangement, an \emph{end} $\epsilon$ of $M$ is determined by a nested sequence $(N_i^{h_i})_{i=1}^{\infty}$ of components of the $N_i$; each component is called a closed \emph{neighborhood} of $\epsilon$ (see Figure \ref{endneighborhood}).
\end{definition}
\begin{remark}
Changing the sequence $\{C_i\}_i$ does not change the homeomorphism type of space of the ends of $M$. More information about ends of spaces can be found in \cite{Freudenthal1931, Freudenthal1942NeuaufbauDE, Peschke1990TheTO}.
\end{remark}
\begin{example}
The Euclidean space $\mathbb{R}^n$ for $n \ge 2$ has one end $\epsilon$. $\epsilon$ has a closed neighborhood that is homeomorphic to $\mathbb{S}^{n-1} \times [0, \infty)$. The space $\mathbb{S}^1 \times \mathbb{R}$ has two ends $\epsilon_1$ and $\epsilon_2$. They have closed neighborhoods $N_1$ and $N_2$ respectively which are both homeomorphic to $\mathbb{S}^1 \times [0,\infty)$. A common example of an infinite-ended
space is the universal cover of $\mathbb{S}^1 \vee \mathbb{S}^1$.
\end{example}

\begin{figure}[htb]
\centering
\begin{tikzpicture}
\node[anchor=south west,inner sep=0] at (0,0){\includegraphics[width=12cm]{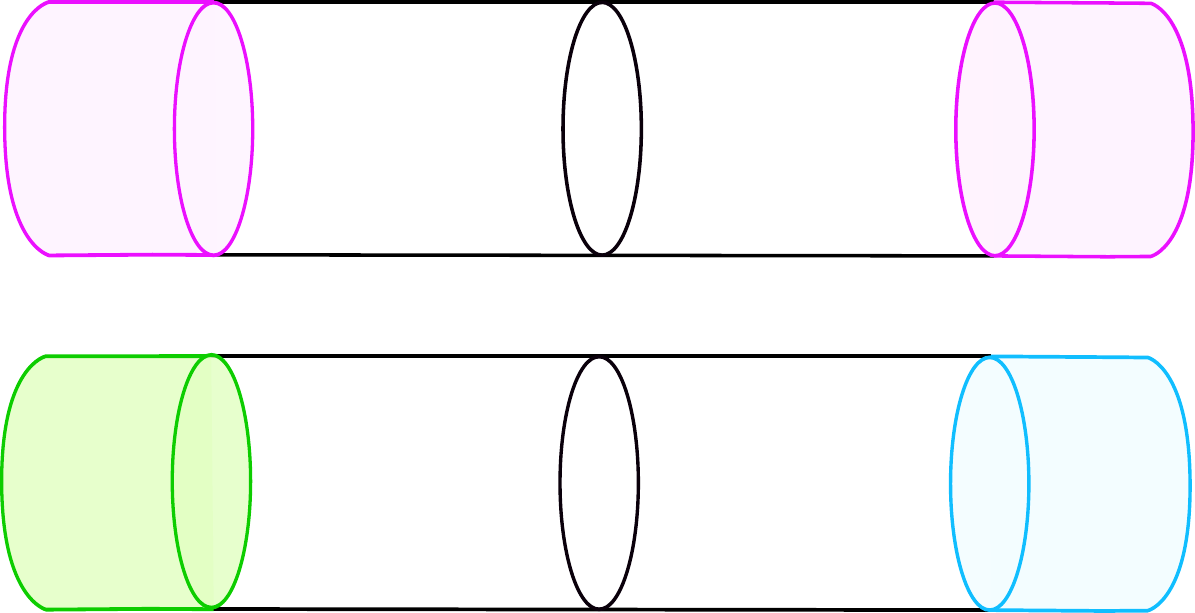}};
\node at (1,1.3) {$N_1$};
\node at (11,1.3) {$N_2$};
\node at (1,4.9) {$U$};
\node at (11,4.9) {$U$};
\node at (6,3.1) {$\mathbb{S}^1\times \mathbb{R}$};
\node at (-0.2,1.3) {$\epsilon_1$};
\node at (12.2,1.3) {$\epsilon_2$};
\end{tikzpicture}
\caption{On the top, $U$ is a closed neighborhood of infinity of $\mathbb{S}^1\times \mathbb{R}$;
$U$ is homeomorphic to the disjoint union of two copies of $\mathbb{S}^1\times [0, \infty)$. 
On the bottom, two closed neighborhoods $N_1$ and $N_2$, both homeomorphic to $\mathbb{S}^1\times [0, \infty)$, of the ends $\epsilon_1$ and $\epsilon_2$ respectively.}
\label{endneighborhood}
\end{figure}
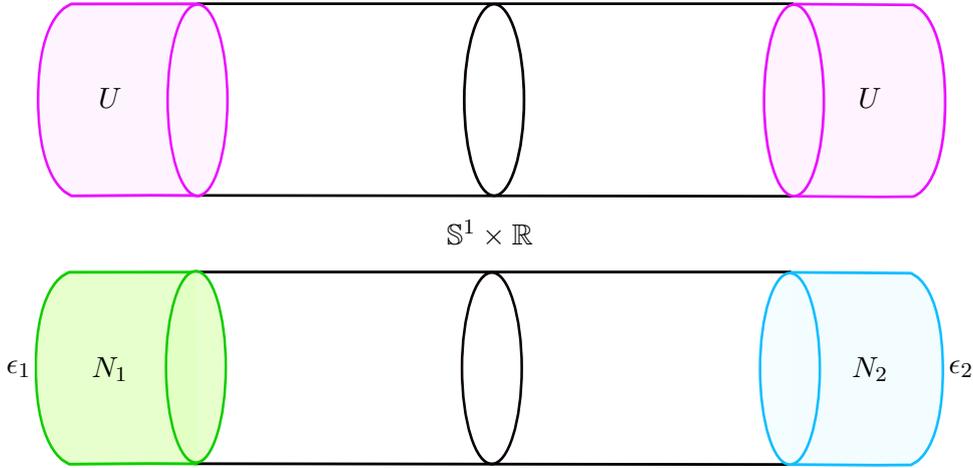
\begin{definition}\label{diffeoinfinity}
Given non-compact smooth manifolds $V$ and $V'$, suppose $f_i \colon Y_i \to Y_i', i = 1,2 $, are diffeomorphisms between closed neighborhoods of infinity $Y_i \subseteq V$ and $Y_i' \subseteq V'$. We refer to $f_1$ and $f_2$ as \emph{equivalent} if they agree outside some compact subset of $V$ containing the complements of $Y_1$ and $Y_2$. A \emph{diffeomorphism at infinity} from $V$ to $V'$ is an equivalence class of such diffeomorphisms. 

If $V$ has a single end, this is referred to as a \emph{diffeomorphism of the end} of $V$ and $V'$. In order to simplify the notation we write $f$ instead of $[f]$ for the equivalence class.

We say that a diffeomorphism at infinity $f$ \emph{extends over} $V$ if the equivalence class $f$ contains a diffeomorphism $V \to V'$. We can compose such diffeomorphisms at infinity, making the set of self-diffeomorphism of $V$ at infinity into a group. Given a group $G$, we define a \emph{$G$-action at infinity} on $V$ to be a homomorphism of $G$ into this group.
\end{definition}

\begin{definition}
Let $V$ be a non-compact smooth manifold. Two diffeomorphisms at infinity of $V$ will be called \emph{isotopic} if they have representatives that are properly isotopic as proper embeddings into $V'$. Recall a \emph{proper} map has preimage of each compact subset compact.

 The group $\mathcal{D}^{\infty}(V)$ of isotopy classes of self-diffeomorphisms of $V$ at infinity is called the \emph{diffeotopy group of $V$ at infinity}.
\end{definition}

We obtain a homomorphism $r\colon \mathcal{D}(V) \to \mathcal{D}^{\infty}(V)$ by sending diffeomorphisms to their corresponding equivalence classes at infinity.

Focusing on closed neighborhoods of the ends instead of closed neighborhoods of infinity, we obtain an analogue of Definition \ref{diffeoinfinity}.
\begin{definition}\label{diffeoend}
Let $V$ and $V'$ be non-compact smooth manifolds and  $\epsilon$ and $\epsilon'$ ends of $V$ and $V'$ respectively.
Suppose $f_i \colon N_i \to N_i', i=1,2 $, are diffeomorphisms between closed neighborhoods of $\epsilon$ and $\epsilon'$. Assume $N_i$ and $N_i'$ are components of $U_i$ and $U_i'$, which are closed neighborhoods of infinity of $V$ and $V'$, respectively. We refer to $f_1$ and $f_2$ as \emph{equivalent} if they agree outside some suitably large compact subset of $V$ containing the complements of $U_1$ and $U_2$. A \emph{diffeomorphism of the end} $\epsilon$ is an equivalence class of such diffeomorphisms. 

We say that a diffeomorphism $f$ of $\epsilon$ \emph{extends over} $V$ if the equivalence class $f$ contains a diffeomorphism $V \to V'$. We can compose such diffeomorphism of the end, making the set of self-diffeomorphism of $\epsilon$ into a group. Given a group $G$, we define a \emph{$G$-action at the end} to be a homomorphism of $G$ into this group. 
\end{definition}
As before, we will denote the equivalence class with $f$ instead of $[f]$.
Note that Definition \ref{diffeoend} fits with Definition \ref{diffeoinfinity} in the case that $V$ has a single end.

\begin{definition}\label{end-diffeotopy}
Let $V$ be a non-compact smooth manifold. Two diffeomorphisms of an end $\epsilon$ of $V$ will be called \emph{isotopic} if they have representatives that are properly isotopic as proper embeddings in $V'$.

 The group $\mathcal{D}^{\epsilon}(V)$ of isotopy classes of self-diffeomorphisms of $\epsilon$ is called the \emph{diffeotopy} group of $\epsilon$.
\end{definition}

For each end of the manifold $V$, we also obtain an  homomorphism $r_{\epsilon}: \mathcal{D}(V) \to \mathcal{D}^{\epsilon}(V)$ by sending diffeomorphisms to their corresponding equivalence class with respect to the end $\epsilon$. 

The following two remarks assume that the ends of $V$ are isolated, which is a consistent assumption in our setting. If we drop this condition, the constructions in remarks \ref{remarksubgroup} and \ref{remarkdirectproduct} might not be well defined.

\begin{remark}\label{remarksubgroup}
We will later refer to $\mathcal{D}^{\epsilon}(V)$ as a subgroup of $\mathcal{D}^{\infty} (V)$. Strictly speaking, this is not quite correct. A diffeomorphism $\phi$ of $\epsilon$ embeds just a closed neighborhood $U$ of the specific end $\epsilon$. However, $\phi$ can be turned into a diffeomorphism at infinity $\overline{\phi}$ by extending it via the identity. Let $U$ be a component of $W$, a closed neighborhood of infinity of $V$. Let $\overline{\phi}$ be the proper embedding of $W$ in $V$ defined by $\overline{\phi}|_{U}= \phi$ and $\overline{\phi}= $ identity on the other components of $W$. When we refer to $\phi$ as a diffeomorphism at infinity we mean $\overline{\phi}$.  Under this identification we can interpret $\mathcal{D}^{\epsilon}(V)$ as a subgroup of  $\mathcal{D}^{\infty} (V)$.
\end{remark}
\begin{remark}\label{remarkdirectproduct}
If $V$ has ends $\epsilon_i$ for $i \in \Lambda$, then the direct product $\prod_{i \in \Lambda} \mathcal{D}^{\epsilon_i}(V)$ can be interpreted as a subgroup of $\mathcal{D}^{\infty}(V)$ as follows. Take a product $(\phi_i)_{i \in \Lambda}$, where $\phi_i$ embeds a closed neighborhood $U_i$ of the end $\epsilon_i$ inside $V$ for each $i$. Then $(\phi_i)_{i \in \Lambda}$ defines a diffeomorphism $\phi$ at infinity of $V$: take a closed neighborhood of infinity $U$, where $U_i$'s are components of $U$,  and define $\phi$ to embed $U_i$ via $\phi_i$ and the remaining components via the identity. Identifying $(\phi_i)_{i \in \Lambda}$ with $\phi$ we turn $\prod_{i \in \Lambda} \mathcal{D}^{\epsilon_i}(V)$ into a subgroup of $\mathcal{D}^{\infty}(V)$. 
\end{remark}
\subsection{Group actions} We follow \cite{groupactions} for the terminology.

Let $\mathbb{R}^4$ denote the Euclidean space with its  standard smoothing.
Let $\Sigma$ be a countable, discrete set and let $\Gamma\colon \Sigma \times [0,\infty) \hookrightarrow \mathbb{R}^4$ be an injection such that $\gamma_s \colon = \Gamma|_{\{s\}\times [0,\infty)}$ is a ray for each $s \in \Sigma$. Let $G$ be a group with the discrete topology.
\begin{definition}{\cite[Definition 4.1]{groupactions}}\label{gamma-compatible}
Fix $\Gamma$ and $G$ as above.
A $G$-action via diffeomorphisms on $\mathbb{R}^4$ will be called \emph{$\Gamma$-compatible} if $G$ acts effectively on $\Sigma$ so that for each $g \in G$ and $(s,t) \in \Sigma \times [0,\infty)$ we have $g \circ \Gamma(s,t) = \Gamma(g(s),t)$ and so that the stabilizer of each $s$ fixes pointwise a neighborhood of $\gamma_s([0,\infty))$.
\end{definition}
Let $\mathcal{G}$ be the set of all groups $G$ so that, after some choice of $\Gamma$, $G$ acts $\Gamma$-compatibly on $\mathbb{R}^4$. $\mathcal{G}^+$ denotes the subset for which the action can be chosen to preserve orientation, and $\mathcal{G}^* \subseteq \mathcal{G}^+$ is the subset for which there is such a $\Gamma$-compatible action fixing pointwise a neighborhood of some ray $\gamma$ in $\mathbb{R}^4$ whose image is disjoint from that of $\Gamma$. We will mainly be interested in the class $\mathcal{G}^* \subsetneq \mathcal{G}$.
\begin{example}\label{exampleG}
\begin{enumerate}
    \item $\mathbb{Q}^3$ lies in $\mathcal{G}^*$. Take $\Gamma$ to be the inclusion $\mathbb{Q}^3 \times [0,\infty) \subset \mathbb{R}^3 \times \mathbb{R} = \mathbb{R}^4 $ where $\mathbb{Q}^3$ is given the discrete topology. Take the obvious action on $\mathbb{R}^3 \times [0,\infty)$, tapered to be trivial on $\mathbb{R}^3 \times (-\infty, -1]$. This defines a $\Gamma$-compatible action of $\mathbb{Q}^3$ on $\mathbb{R}^4$. By \cite[Proposition 4.2]{groupactions},  many uncountable groups such as $\mathbb{Q}^{\omega}$, where $\omega$ denotes the first infinite cardinal number, lie in $\mathcal{G}^{*} \subsetneq \mathcal{G}^{+} \subsetneq \mathcal{G}$. 
    \item Every countable group $G$ of Euclidean or hyperbolic isometries of $\mathbb{R}^3$ lies in $\mathcal{G}$, and if it preserves orientation, in $\mathcal{G}^+$. Choose a point $p \in \mathbb{R}^3$ with trivial stabilizer under the elements $g \in G$. Define then $\Gamma : G \times [0, \infty) \to \mathbb{R}^3 \times \mathbb{R} = \mathbb{R}^4$ by setting $\Gamma(g,t)=(g(p),t)$.
\end{enumerate}
\end{example}

\subsection{Main results from \cite{groupactions}}
We next state the results of \cite{groupactions} that we will generalize.
\begin{lemma}[{\cite[Theorem 4.4, Case (1)]{groupactions}}]\label{Gompf1}
There are uncountably many diffeomorphism types of $\mathbb{R}^4$ homeomorphs $\mathcal{R}$\footnote{An $\mathbb{R}^4$ homeomorph is a manifold homeomorphic to $\mathbb{R}^4$ but not necessarily diffeomorphic to it. We follow\cite{groupactions} for the terminology.} satisfying the following characterization. $\mathcal{R}$ embeds in $\mathbb{R}^4$ and every $G \in \mathcal{G}$ has an action on $\mathcal{R}$ that injects into the diffeotopy groups $\mathcal{D}(\mathcal{R})$ and $\mathcal{D}^{\infty}(\mathcal{R})$. The elements of $G$ define orientation preserving diffeomorphisms of $\mathcal{R}$ if and only if the defining $G$-action on $\mathbb{R}^4$ preserves orientation.
\end{lemma}
We briefly discuss the construction in the proof of \cite[Theorem 4.4]{groupactions}; this will be useful in order to better understand our own construction in Section 3.
\begin{const}\label{constr}
This is the construction of an $\mathbb{R}^4$ homeomorph $\mathcal{R}$ satisfying the characterization in Lemma \ref{Gompf1}. We do not provide the proof for the injection property; it can be found in \cite{groupactions}.

We will build $\mathcal{R}$ as an end sum.\footnote{This is the non-compact analogue of the boundary sum. Given two smoothings $\mathcal{R}_1, \mathcal{R}_2$ of $\mathbb{R}^4$, the end sum $\mathcal{R}_1 \natural \mathcal{R}_2$ is defined as follows: let $\gamma_1: [0,\infty) \to \mathcal{R}_1$ and $\gamma_2: [0,\infty) \to \mathcal{R}_2$ be two smooth properly embedded rays with tubular neighborhoods $\nu_1, \nu_2$ respectively. Then $\mathcal{R}_1 \natural \mathcal{R}_2:= \mathcal{R}_1 \cup_{\phi_{1}} ([0,1] \times \mathbb{R}^3) \cup_{\phi_{2}} \mathcal{R}_2$ , where $\phi_1: [0,\frac{1}{2}) \times \mathbb{R}^3 \to \nu_1 $ and $\phi_2: (\frac{1}{2},1] \times \mathbb{R}^3 \to \nu_2 $ are orientation preserving diffeomorphisms which respect the $\mathbb{R}^3$ bundle structures. For a more comprehensive treatment of the end sum operation, see \cite{end-sum}.}
For a fixed $G \in \mathcal{G}$, $\Sigma$ and $\Gamma$ (as in Definition \ref{gamma-compatible}), where $G$ acts $\Gamma$-compatibly on $\mathbb{R}^4$, identify $\Sigma$ with $\mathbb{Z}^{+}$. We will find an increasing sequence $(t_n)$ such that, when each ray $\gamma_n:= \Gamma|_n$ is restricted to the interval $[t_n, \infty)$, we get a multiray $\overline{\Gamma}$ with an equivariant\footnote{By equivariant we mean $g \circ \varphi_n(x) = \varphi_{g(n)}(x)$.} tubular neighborhood map, whose restriction over each truncated ray we denote by $\varphi_n \colon [t_n, \infty) \times \mathbb{R}^3 \to \mathbb{R}^4$.

We truncate the rays $\gamma_n$ for each $n$ so that they can be tubed without getting undesired intersections.

To do so, construct the sequence by induction on $n$, starting with $t_1=0$. Given $n \ge 1$, suppose that for each $i \le n-1$ we have already equivariantly defined each $\varphi_i|_{[t_i,t_n] \times \mathbb{R}^3}$ so that the image lies in a compact set $K_{in}$ disjoint from each other $K_{jn}$ and from each $\gamma_{j}((t_n, \infty))$ with $j \le n$. Choose $t_{n+1} \ge t_n + 1$ so that $\gamma_{n+1} ([t_{n+1}, \infty))$ avoids each $K_{in}$ and the ball of radius $n$. Then define $\varphi_{i}|_{[t_n,t_{n+1}]\times \mathbb{R}^3}$ for each $i \le n$ so that the induction hypotheses are extended to $n+1$.

If we end sum $\mathbb{R}^4$ with copies of a fixed  $\mathbb{R}^4$ homeomorph along each $\varphi_n$, the resulting $\mathbb{R}^4$ homeomorph,  call it $\mathcal{R}$, inherits a $G$-action. 
\end{const}
\begin{remark}\label{gammainR}
Start with $G \in \mathcal{G}^*$ with a $\Gamma$-compatible action on $\mathbb{R}^4$ for some multiray $\Gamma$. By definition, there is a ray $\gamma$ in $\mathbb{R}^4$, whose image is disjoint from $\Gamma$, with a $G$-fixed neighborhood. Since in Construction \ref{constr} we used end-sums along the (truncated) rays $\gamma_n$ to construct  $\mathcal{R}$, $\gamma$ defines a proper embedding of $[0, \infty)$ in $\mathcal{R}$ as well. Then the $G$-action inherited by $\mathcal{R}$ fixes pointwise $\gamma$ and its neighborhood. 
\end{remark}

\section{Proof of Theorem \ref{new}}
The goal of this section is to extend  Lemma \ref{Gompf1}; how far can we extend this result to general punctured 4-manifolds? The starting point will be the case of a once punctured \emph{smooth, closed} 4-manifold $M$ (Lemma \ref{theorem1}). This first step will give us the construction we will use to prove Theorem \ref{new} in full generality.

\begin{lemma} \label{theorem1}
Let $M$ be a closed smooth 4-manifold and $p \in M$. Then there exists a smoothing  of $M \setminus p$ whose diffeotopy groups  $\mathcal{D}^{\infty}(M \setminus p)$ and $\mathcal{D}(M \setminus p)$ are uncountable.
\end{lemma}
\begin{proof}
We start by choosing an $\mathbb{R}^4$ homeomorph $\mathcal{R}$. For every uncountable group $G \in \mathcal{G}^*$ there exists (Lemma \ref{Gompf1}) an $\mathbb{R}^4$ homeomeorph, call it $\mathcal{R}$, with a $G$-action fixing pointwise a neighborhood of some ray $\gamma$ (see Remark \ref{gammainR}). $G$ injects in the diffeotopy groups $\mathcal{D}(\mathcal{R})$ and $\mathcal{D}^{\infty}(\mathcal{R})$ by Lemma \ref{Gompf1}.

$M \# \mathbb{R}^4$ is homeomorphic to $M \setminus p$; smooth it as the smooth connected sum between $M$ and $\mathcal{R}$. This gives a smoothing of $M\setminus p$.

 As stated before $G$ acts on $\mathcal{D}(\mathcal{R})$ and $ \mathcal{D}^{\infty}(\mathcal{R})$. Call the group homomorphism $\theta$, i.e each $g \in G$ defines a self-diffeomorphism $\theta(g)$ of $\mathcal{R}$.\footnote{In Definition \ref{diffeoinfinity} we defined a $G$-action at infinity as a homomorphism of $G$ into the diffeomorphism group at infinity. Composing this with the quotient map, we obtain a homomorphism into the diffeotopy group at infinity. In order to simplify the notation, we will omit the quotient map. When we say that a $G$-action at infinity defines a diffeomorphism at infinity we usually mean an isotopy class of diffeomorphisms at infinity. We do the same for $G$-actions via diffeomorphisms.}
 We extend the $G$-action to an action on $\mathcal{D}(M\setminus p)$ and $\mathcal{D}^{\infty}(M\setminus p)$. Each $g$ will  define a self-diffeomorphism of the punctured manifold. 
To do so, take a disk $D$ inside the $G$-fixed neighborhood of $\gamma$; use this disk to perform the smooth connected sum with $M$. Let $g \in G$. Since the $G$-action is trivial in $D$, the diffeomorphism $\theta(g)$ is the identity on $D$. We can then extend it via the identity to $M^4 \setminus \mathring{D'}$, where $D'$ is a disk in $M$ around the point $p$. Hence for each $g \in G$  we get a diffeomorphism of $M \setminus p$  which is the identity on $M \setminus \mathring{D'}$. 

This defines a $G$-action on $\mathcal{D}(M\setminus p)$ (and, by post-composing it with $r$, also on $\mathcal{D}^{\infty}(M\setminus p)$). Call this action $\widetilde{\theta}$; the diffeomorphism defined by $g \in G$ is then $\widetilde{\theta}(g)$.
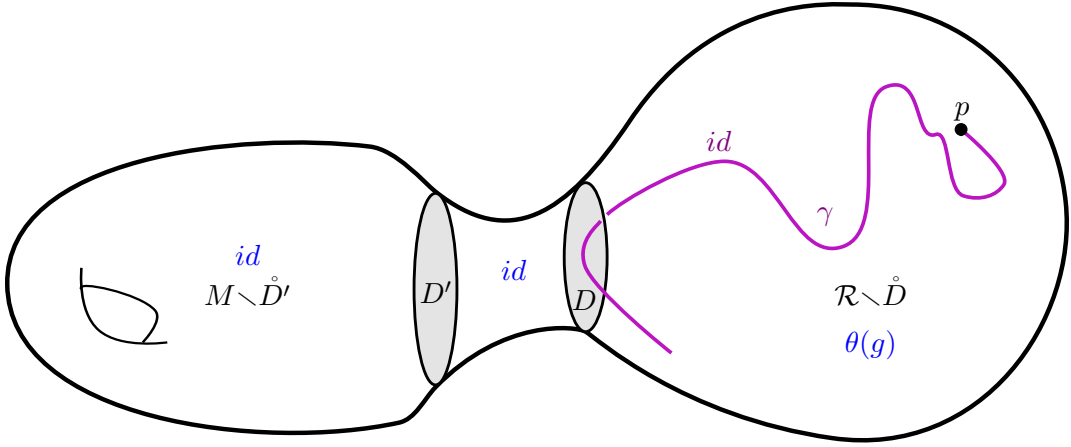
\begin{figure}[htb]
\centering
\begin{tikzpicture}
\node[anchor=south west,inner sep=0] at (0,0){\includegraphics[width=14cm]{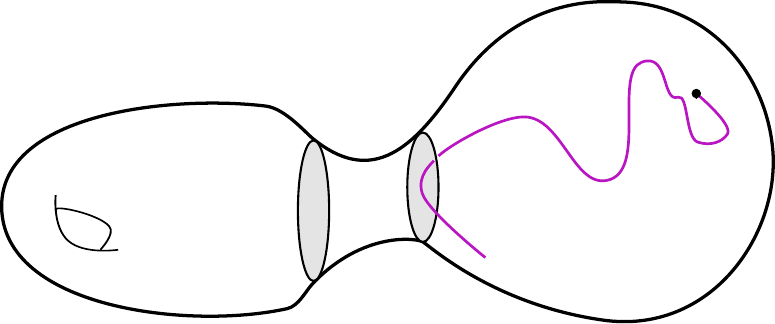}};
\node at (3.2,2) {$M \setminus \mathring{D'}$};
\node at (10.8,3) {$\textcolor{violet}{\gamma}$};
\node at (11.4,2) {$\mathcal{R} \setminus \mathring{D}$};
\node at (3.2,2.5) {$\textcolor{blue}{id}$};
\node at (6.7,2.3) {$\textcolor{blue}{id}$};
\node at (9.4,4) {$\textcolor{violet}{id}$}; 
\node at (11.4,1.3) {$\textcolor{blue}{\theta(g)}$};
\node at (12.6,4.4) {$p$};
\node at (5.68,2){$D'$};
\node at (7.63,1.9){$D$};
\end{tikzpicture}
\caption{The construction of the diffeomorphism $\widetilde{\theta}(g)$. The $G$-fixed ray $\gamma$ going to infinity is in violet. In light gray, the two 4-disks $\mathring{D'}$ and $\mathring{D}$ used to perform the connected sum. In blue we indicate the definition of $\widetilde{\theta}(g)$ on $M \# \mathcal{R}$: $\theta(g)$ is extended via the identity on the boundaries of the disks. }
\label{Extending}
\end{figure}
We give the above construction, which assigns for each $\theta(g)$ diffeomorphism of $\mathcal{R}$ the diffeomorphism $\widetilde{\theta}(g)$ of $M\setminus p$, the name $\#$, reminding us that the connected sum is used to define it. So $\#(\theta(g)) = \widetilde{\theta}(g)$ is defined as $\widetilde{\theta}(g)|_{\mathcal{R} \setminus D} = \theta(g)|_{\mathcal{R} \setminus D}$ and $\widetilde{\theta}(g)|_{M \setminus D'}= id$. See Figure \ref{Extending} for a visual representation of $\#$.

Define $\phi$ (resp. $ \widetilde{\phi}$) to be the composition $r| \circ \theta$ (resp. $r| \circ \widetilde{\theta}$) where $r \colon \mathcal{D}(\mathcal{R}) \to \mathcal{D}^{\infty}(\mathcal{R})$ is the homomorphism we defined in Section 2.\footnote{By $r|$ we mean the restriction of $r $ on $\theta(G)$ (resp. $\widetilde{\theta}(G)$).}

Finally define $\#^{\infty}$ to be the homomorphism sending an equivalence class $\phi(g) \in \phi(G) \subseteq \mathcal{D}^{\infty}(\mathcal{R})$ to the corresponding equivalence class $\widetilde{\phi}(g)$ in $\widetilde{\phi}(G) \subseteq\mathcal{D}^{\infty}(M \setminus p)$. A representative of $\phi(g)$ is given by the embedding of some closed neighborhood of infinity $U$ defined as $\theta(g)|_U$. Shrink $U$ if necessary to make it disjoint from $D$. Then $\theta(g)|_U$ embeds $U$ in $M \setminus p$ as well. In particular, the restriction of $\theta(g)$ on $U$ is the same as the restriction of $\widetilde{\theta}(g)$ on $U$ by our connected sum construction. So $\#^{\infty} (\phi(g))$ is the equivalence class of the diffeomorphism at infinity given by the embedding $\widetilde{\theta}(g)|_U$ of $U$ in $M \setminus p$, i.e the equivalence class $\widetilde{\phi}(g)$.

We summarize everything we have obtained so far in the following diagram:
\begin{center}
\begin{tikzcd}

&\mathcal{D}(\mathcal{R}) \arrow[r, "r"] & \mathcal{D}^{\infty}(\mathcal{R})\\
G \arrow[rd, "\widetilde{\theta}"', bend right=35] \arrow[r, "\theta"] & \theta(G) \arrow[d, "\#"] \arrow[r, "r|"] \arrow[u,symbol=\subseteq] & \phi(G) \arrow[d, "\#^{\infty}"] \arrow[u,symbol=\subseteq]  \\
                                                                       & \widetilde{\theta}(G) \arrow[r, "r|"] \arrow[d,symbol=\subseteq]    & \widetilde{\phi}(G) \arrow[d,symbol=\subseteq] \\
                                              &\mathcal{D}(M \setminus p) \arrow[r, "r"] & \mathcal{D}^{\infty}(M \setminus p)                         
\end{tikzcd}
\end{center}

Gompf proved \cite[Theorem 4.4]{groupactions} that the map $\phi$ (hence also $\theta$) in the upper row is injective, i.e the $G$-action defines an injection of $G$ in $\mathcal{D}^{\infty}(\mathcal{R})$ (hence and also in $\mathcal{D}(\mathcal{R})$) (see Lemma \ref{Gompf1} above). By our definitions of $\widetilde{\theta}, \widetilde{\phi}, \# $ and $\#^{\infty}$, both the left triangle and the right square in the diagram commute.

We want to prove that $\widetilde{\phi}= r \circ \widetilde{\theta}$ is an injection. To do so, it suffices to show that $\#^{\infty}$ is injective, since by commutativity of the diagram  $\widetilde{\phi}= \#^{\infty} \circ \phi$, and we already know that $\phi $ is injective.

\begin{claim}
The group homomorphism $\#^{\infty}$ constructed as above is injective.
\end{claim}
\begin{proof}[Proof of the claim]
For ease of notation, throughout the proof we call $\phi(g)=:g$, meaning that this is the class of diffeomorphisms of infinity that behave like the diffeomorphism $\theta(g)$ on some closed neighborhood of infinity $U$.

If $ \#^{\infty}(g) = id  \in \mathcal{D}^{\infty}(M \setminus p)$, then there are two representatives of their equivalence classes that are properly isotopic as embeddings of $U$, a closed neighborhood of infinity in $M\setminus p$. By our definition of $\#^{\infty}(g)$, a representative is given by the restricion of $\widetilde{\theta}(g)$ on some closed neighborhood of infinity of $M \# \mathcal{R}$ entirely contained in $\mathcal{R} \setminus \mathring{D}$.
Then shrinking $U$ if necessary, we can assume that the two representatives of $\widetilde{\phi}(g)$ and $id$ which are \emph{properly isotopic} embed $U$ in $M\setminus p$  as $\theta(g)|_U$ and as $id$. So we have $H: U \times [0,1] \to M \setminus p$ a proper isotopy with $H_0 = \theta(g)|_U$ and $H_1= id$.

We wish to show that there is such a proper isotopy between representatives of $g$ and $ id \in \mathcal{D}^{\infty}(\mathcal{R})$, i.e a proper isotopy $G: U' \times [0,1] \to \mathcal{R}$ between embeddings $\theta(g)|_{U'}$ and $id$ of $U'$ a closed neighborhood of infinity of $\mathcal{R}$.

We can consider the track $H$ of the proper isotopy, i.e the (level preserving) proper embedding $H: U\times [0,1] \hookrightarrow (M\setminus p) \times [0,1] $. Let $K:= \partial D'$ (the boundary of the disk used to perform the gluing) and $C:= H (U\times [0,1]) \cap (K \times [0,1])$ (see Figure \ref{Injectivity}). $U\times [0,1]$ is closed; $H$ is proper and hence a closed map, which means $H(U \times [0,1])$ is closed as well. $K \times [0,1]$ is compact, so $C$ is compact as it is the intersection between a closed and a compact set.

Take the preimage $H^{-1}(C)$, which is a compact subset of $U \times [0,1]$. Exhaust $U\times [0,1]$ by compact sets $K_n \times [0,1]$; since $H^{-1}(C)$ is compact, it is contained in some $K_N \times [0,1]$, $N \in \mathbb{N}$. Henceforth, we can find a closed (connected) neighborhood of infinity $U' \subseteq U$ so that $U'\times [0,1]$ is disjoint from $H^{-1}(C)$ (see Figure \ref{Injectivity}). Then the image of $U'\times [0,1]$ under $H$ is entirely contained in $\mathcal{R} \setminus \mathring{D}$, since the intersection with $\partial D'$ is trivial and $H(U'\times[0,1])$ is connected. So the restriction of $H$ to $U'$ defines a proper isotopy between $\theta(g)|_{U'}$ and $id$ as embeddings of $U'$ in $\mathcal{R}$. This means that $g=id$ as equivalence classes in $\mathcal{D}^{\infty}(\mathcal{R})$.
\begin{figure}[htb]
\centering
\begin{tikzpicture}
\node[anchor=south west,inner sep=0] at (0,0){\includegraphics[width=14cm]{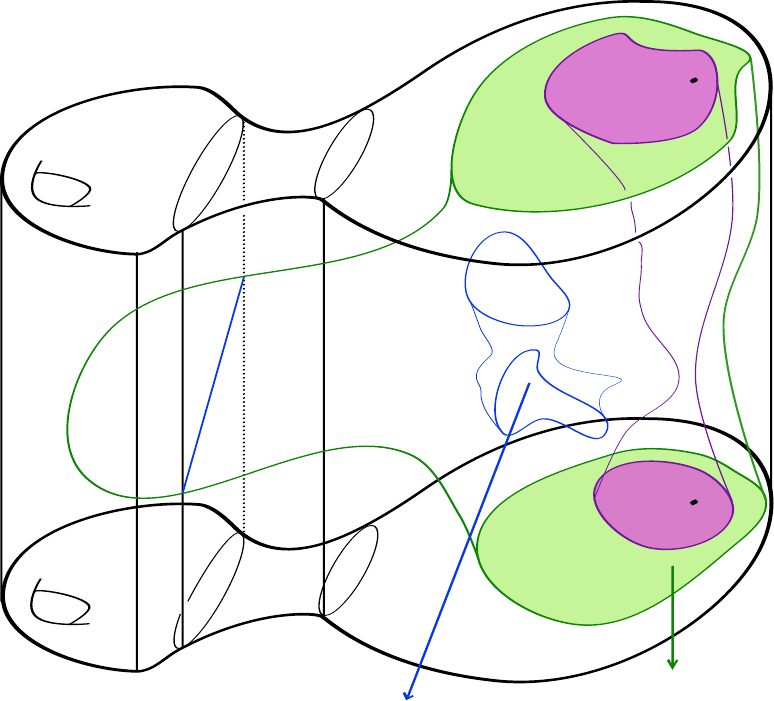}};
\node at (2.6,11.45) {$M \setminus \mathring{D'}$};
\node at (10.6,13) {$\mathcal{R} \setminus \mathring{D}$};
\node at (4.1,5.5) {$\textcolor{blue}{C}$};
\node at (11.4,11) {$\textcolor{violet}{U'}$};
\node at (10.6,9.6) {$\textcolor{Green}{U}$}; 
\node at (7.25,-0.3) {$\textcolor{blue}{H^{-1}(C)}$};
\node at (11.7,3.6){$\textcolor{violet}{\theta(g)(U')}$};
\node at (12.4,0.3) {$\textcolor{Green}{\theta(g)(U)}$};
\end{tikzpicture}
\caption{A schematic representation of the track $H$. The image $H(U \times [0,1])$ is in green lines. The compact set $C$ is colored in blue; its preimage under $H$ is bounded by blue lines. The closed neighborhood of infinity $U'$ is in violet; the picture shows how $H(U' \times[0,1])$ is entirely contained in $\mathcal{R}\setminus \mathring{D^4}$.}
\label{Injectivity}
\end{figure}
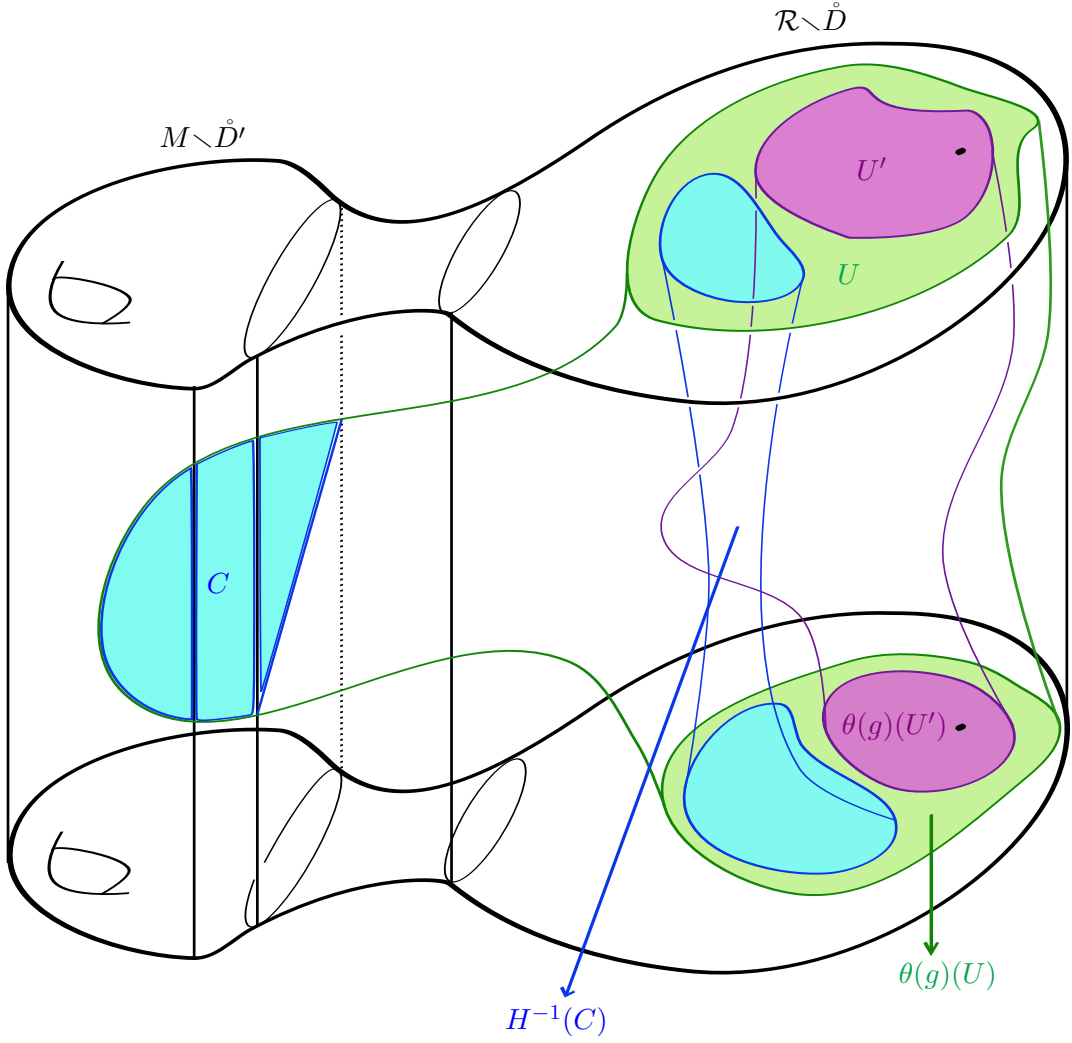
\end{proof}
Using the claim, we get that $\widetilde{\phi}$ is indeed an injection. Clearly, since $\widetilde{\phi}= r \circ \widetilde{\theta}$, the same holds for $\widetilde{\theta}$. This proves that $G$ injects in $\mathcal{D}(M \setminus p ) $ an $\mathcal{D}^{\infty}(M \setminus p)$, showing that both are uncountable.
\end{proof}

We started with once punctured closed, smooth 4-manifolds. The following corollary extends the result of Lemma \ref{theorem1} to give uncountably many smoothings having the desired property.

\begin{corollary} \label{corollary1}
Let $M$ be a closed, smooth  4-manifold and $p \in M$. Then there exist uncountably many smoothings of  $M \setminus p$ whose diffeotopy groups  $\mathcal{D}^{\infty}(M\setminus p)$ and $\mathcal{D}(M\setminus p)$ are uncountable.
\end{corollary}
\begin{proof}
Given an uncountable group $G \in \mathcal{G}^*$ (see Example \ref{exampleG} (1)) we can choose $\mathcal{R}$ from an \emph{uncountable} family of $\mathbb{R}^4$ homeomorphs $\{\mathcal{R}_t\}_t$ by equivariantly varying the model summand in the construction (see \cite[Lemma 3.2]{groupactions}).

Consider the family of smoothings of $M \setminus p$ given by $\{M \# \mathcal{R}_t\}_{t\in \mathbb{R}}$. Since only countably many  $\mathbb{R}^4$ homeomorphs can realize a given end \cite{anexoticmanagerie}, this uncountable family contains an uncountable subfamily of smoothings of $M \setminus p $ with pairwise non-diffeomorphic ends. Hence, this represents an uncountable family of distinct smoothings of $M \setminus p $ so that each one satisfies the properties of the above Lemma \ref{theorem1}.
\end{proof}
We proved the basic case. We can now state the result in full generality; the techniques we will use in the proof are analogous to the ones used in the proofs of Lemma \ref{theorem1} and Corollary \ref{corollary1}.
\begin{theorem}\label{new}
Let $M$ be a topological 4-manifold, $S \subsetneq \mathring{M}$ a non-empty discrete set of points. Then $M\setminus S$ admits uncountably many smoothings whose diffeotopy groups $\mathcal{D}^{\infty}(M\setminus S)$ and $\mathcal{D}(M\setminus S)$ are uncountable if:
\begin{enumerate}
    \item $M$ is smoothable or
    \item $M$ is non-smoothable and $|S| \ge 2$.
\end{enumerate}
\end{theorem}
\begin{remark}\label{noboundary}
We impose the assumption $S \subsetneq \mathring{M}$ in order to avoid removing points on the (possibly non-empty) boundary $\partial M$. Going through the proof of Lemma \ref{theorem1}, we see that we made almost no use of the empty boundary assumption. However if $p \in \partial M$, then the initial idea,  connect-summing with $\mathcal{R}$, needs to be modified. We have to take the boundary sum with $\mathbb{H}^4$, the upper half 4-space. We do not know how the diffeotopy groups of $\mathbb{H}^4$ behave (note that this is $D^4 \setminus pt $ where the point lies in the boundary $\mathbb{S}^3$, a case we will not cover) and so no direct conclusions can be made from this construction.
\end{remark}
\begin{proof}[Proof of Theorem \ref{new}]

We already covered case (1) when $|S|=1$ and $M$ is closed (by Remark \ref{noboundary} also compact with possibly non-empty boundary) in Lemma \ref{theorem1}. To prove Theorem \ref{new}, it suffices to prove case (1) for non-compact 4-manifolds; if $M$ is a \emph{smooth, compact} 4-manifold, $|S| \ge 2$ then apply the result to $M' := M \setminus p$, in which case $M'$ is non-compact, and $S' := S \setminus p$, where $p \in S$. Similarly, if $M$ is a  non-smoothable 4-manifold (2) follows from (1) applied to $M' := M \setminus p$ (by Quinn's \emph{smoothability theorem} \cite{endsofmaps} $M \setminus p$ has a smooth structure) and $S' := S \setminus p$, where $p \in S$.

We first prove (1) for $M$ a non-compact 4-manifold and $|S|=1$. Once we have this, we will show how to obtain the general result. 
\begin{claim}
Let $M$ be a non-compact 4-manifold and $p \in \mathring{M}$. Then there exist uncountably many smoothings of  $M\setminus p$ whose diffeotopy groups  $\mathcal{D}^{\infty}(M\setminus p)$ and $\mathcal{D}(M\setminus p)$ are uncountable. 
\end{claim}
\begin{proof}[Proof of the claim]
The proof follows that of  Lemma \ref{theorem1} with the natural modifications. Note that $K:= \partial D'$ separates $\epsilon_p$ from all the other ends of $M$. 

To show there are uncountably many distinct smoothings of $M \setminus p$ satisfying this property we work as in Corollary \ref{corollary1}. 
Consider the uncountable family of $\mathbb{R}^4$ homeomorphs  $\{\mathcal{R}_t\}_{t\in \mathbb{R}}$ with uncountable diffeotopy groups. Take the uncountable subfamily $\{\mathcal{R}_s\}_{s \in \mathbb{R}}$ having pairwise non-diffeomorphic ends. As in \cite{anexoticmanagerie}, $M$ can only have countably many $\mathbb{S}^3 \times \mathbb{R}$ collared ends. So there is another subfamily $\{\mathcal{R}_u\}_{u\in \mathbb{R}}$, with pairwise non-diffeomorphic ends, so that no end has the same diffeomorphism type of an end of $M$. 
Consider then the family of smoothings $\{M \# \mathcal{R}_u\}_{u \in \mathbb{R}}$. This realizes an uncountable family of distinct smoothings of $M \setminus p $ each one satisfying the desired property for $\mathcal{D}(M \setminus p)$ and $\mathcal{D}^{\infty}(M \setminus p)$. This completes the proof of the claim.
\end{proof}

We have proved (1) when $M$ is non-compact and $|S|=1$; if $S$ is an arbitrary discrete subset, the result follows from the claim. More precisely, let $S$ be discrete and $M$ non-compact. Take $p \in S$, define $S'$ to be $S \setminus p$ and $M':= M \setminus S'$. Then $M'$ is non-compact. By the claim, $M'\setminus p = M \setminus S$ has uncountable $\mathcal{D}(M' \setminus p)$ and $\mathcal{D}^{\infty}(M' \setminus p)$ (since the set $S$ is discrete, we can find $K$ as in the proof of the claim separating the end $\epsilon_p$ from the other punctured ends) and this completes the proof of Theorem \ref{new}.
\end{proof}

\begin{example}
Theorem \ref{new} implies that $\mathbb{S}^4\setminus \{p,q\}$ with $p \neq q \in \mathbb{S}^4$, which is homeomorphic to $\mathbb{S}^3 \times \mathbb{R}$, has uncountably many smoothings whose diffeotopy groups\\ $\mathcal{D}(\mathbb{S}^4\setminus \{p,q\})$ and $\mathcal{D}^{\infty}(\mathbb{S}^4\setminus \{p,q\})$ are uncountable.
\end{example}

In the proof of Theorem \ref{new} we constructed a $G$-action on $\mathcal{D}^{\epsilon_p}(M \setminus p)$. By carefully modifying the construction, we obtain the following corollary.
\begin{corollary}\label{corollary2}
Let $G \in \mathcal{G}^*$ be uncountable. Let $M$ be a topological 4-manifold, $S \subsetneq \mathring{M}$ a non-empty discrete set of points. Call $\Lambda$ the index set of $S$. Let \textbf{p1} and \textbf{p2} be the following properties:
\begin{itemize}
    \item \textbf{p1}: \emph{The direct product $\prod_{i \in \Lambda} G$ injects into the diffeotopy groups $\mathcal{D}(M\setminus S)$ and  $\mathcal{D}^{\infty}(M\setminus S)$}.
    \item \textbf{p2}: \emph{The direct product $G \times G \cdots \times G = G^{|S|-1}$, where $S$ is finite, injects into the diffeotopy groups $\mathcal{D}(M\setminus S)$ and  $\mathcal{D}^{\infty}(M\setminus S)$}.
\end{itemize}

Then there are uncountably many smoothings of $M \setminus S$ for which property \textbf{p1} holds if $M$ is a smoothable 4-manifold.

Moreover, there are uncountable many smoothings of $M \setminus S$ for which property \textbf{p2} holds when $M$ is a non-smoothable compact 4-manifold and $|S| \ge 2$.
\end{corollary}
\begin{proof}
The cases $|S|=1$ for \textbf{p1} and $|S|=2$ for \textbf{p2} follow from Theorem \ref{new}. We just need to prove the \textbf{p1} case for $|S| \ge 2$; the \textbf{p2} case descends from it.

We start with a non-compact 4-manifold $M$. Fix an uncountable $G \in \mathcal{G}^*$. Let $M$ be  non-compact, $S$ discrete, $|S|\ge 2$. The index set $\Lambda$ must be countable otherwise $S$ would not be discrete, so we can identify it with a subset of  $\mathbb{Z}^+$. Start with $p_1 \in S$. In the proof of Theorem \ref{new} we constructed a smoothing of $M \setminus p_1$ with an injection $G \hookrightarrow \mathcal{D}^{\epsilon_{p_1}}(M \setminus p_1)$. Fix this smoothing and call the resulting smooth manifold $M'$. Now take $p_2 \in S$ and smooth $M'\setminus p_2$ as in the proof of Theorem \ref{new} so that $G$ injects into $\mathcal{D}^{\epsilon_{p_2}}(M \setminus \{p_1, p_2 \})$. Since $S$ is discrete we can choose the disk $D'$ away from $p_1$ to perform the connected sum with $\mathcal{R}$. The smooth structure is not changed near the puncture $p_1$. Hence 
$G$ injects in  $\mathcal{D}^{\epsilon_{p_1}}(M \setminus \{p_1, p_2\})$ as well.

Now continue the procedure for each index $i \in \Lambda$. In the end we get a smoothing of $M \setminus S$ so that $G$ injects into $\mathcal{D}^{\epsilon_{p_i}}(M \setminus S)$ for each $i \in \Lambda$. The direct product $\prod_{i \in \Lambda} \mathcal{D}^{\epsilon_{p_i}}(M \setminus S)$ is a subgroup of $\mathcal{D}^{\infty}(M \setminus S)$ (see Remark \ref{remarkdirectproduct}). Hence we get that the direct product $\prod_{i \in \Lambda} G$ injects into $\mathcal{D}^{\infty}(M \setminus S)$.
If $M$ is compact and smoothable, the proof is analogous.

As for \textbf{p2}, start with $M$ a non-smoothable compact 4-manifold, $|S| \ge 3$. Take $p \in S$ and define $M':= M \setminus p$, $S'= S \setminus p$. Then $M'$ is non-compact and $|S'| \ge 2$. Apply the case of \textbf{p1} to $M'$ and $S'$ to get that $G^{|S'|}$ injects into $\mathcal{D}(M'\setminus S')$ and $\mathcal{D}^{\infty}(M'\setminus S')$. Note that $|S'| = |S|-1$. Then use the equality $M' \setminus S' = M \setminus S$.
To get uncountability in both cases, choose different $\mathcal{R}$ summands in the construction as we did in the proof of Theorem \ref{new}. Thus we can conclude.
\end{proof}

\begin{remark}
If we take in consideration different uncountable groups $G_i \in \mathcal{G}^{*}$ for $i \in \Lambda$, where $\Lambda $ is the index set of $S$, then we can use the proof of Corollary \ref{corollary2} combining more group actions together. This would show that for $\{G_i\}_i$ a family in $\mathcal{G}^*$, the same results as in Corollary \ref{corollary2} apply for the  product $\prod_{i\in \Lambda} G_i$.
\end{remark}

\section{Proof of Theorem \ref{new1}}

The goal of this section is to extend the following result to general punctured 4-manifolds.

\begin{lemma}[{\cite[Theorem 5.1]{groupactions}}]\label{Gompf2}
Every diffeomorphism of the end of $\mathcal{R}_U$ extends over $\mathcal{R}_U$. That is, the homomorphism $r \colon \mathcal{D}(\mathcal{R}_U) \to \mathcal{D}^{\infty}(\mathcal{R}_U)$ is surjective.
\end{lemma}

We will be able to achieve a similar statement; in the case of non-compact 4-manifolds, however, we shall restrict only to diffeomorphisms of the punctured ends since the other ends might have a wild behavior that we cannot detect with the tools we have developed so far. As an example, there might be an end that is not topologically collared by any $M \times \mathbb{R}$, where $M$ is a compact 3-manifold.

As we did for Lemma \ref{theorem1}, we begin with the case of smooth, compact 4-manifolds. The starting point of the proof will be again a connected sum argument. As we previously saw (see Remark \ref{noboundary}), as long as the point is assumed to be taken out from the interior of the manifold, we can allow the manifold $M$ to have non-empty boundary.

In the following discussion, if $M$ is orientable, then we have to restrict to the orientation-preserving diffeomorphism of the ends, otherwise the \emph{Cerf-Palais disk theorem} \cite{Cerf,Palais} (which we use in the proof of Lemma \ref{theorem3}) does not hold. We will however denote the groups by the usual $\mathcal{D}$, suppressing the $_+$ subscript, to ease the notation.

\begin{lemma}\label{theorem3}
Let $M$ be a smooth compact 4-manifold and $p \in \mathring{M}$. Then there is a smoothing of $M \setminus p$ so that every diffeomorphism at infinity of $M\setminus p$ extends over $M\setminus p$. That is, the homomorphism $r \colon \mathcal{D}(M\setminus p) \to \mathcal{D}^{\infty}(M\setminus p)$ is surjective.
\end{lemma}
\begin{proof}
 Smooth $M \setminus p$ as the smooth connected sum $M \# \mathcal{R}_U$, where $\mathcal{R}_U$ is Freedman and Taylor's universal $\mathbb{R}^4$ \cite{universal}.
 We want to show that  $r \colon \mathcal{D}(M \setminus p) \to \mathcal{D}^{\infty}(M\setminus p)$ is surjective.

Consider $\phi \in \mathcal{D}^{\infty}(M \setminus p)$ and pick $\phi$ as a representative of the class. Then $\phi$ is a proper embedding of $V$, a closed neighborhood of infinity of $M \setminus p$. By shrinking $V$ if necessary, we can assume that it lies entirely in the $\mathcal{R}_U\setminus \mathring{D}$ part. Call $\phi(V)=: V'$.

Our goal is to restrict $\phi$ to a diffeomorphism at infinity of $\mathcal{R}_U$. At present, $V'$ is a closed neighborhood of infinity of $M \setminus p$, but it might intersect the $M^4\setminus \mathring{D'}$ part non-trivially. 

As in Lemma \ref{theorem1}, let $K:= \partial D'$ be the boundary of the gluing disk. $V' \cap K$ is compact, call it $C$. $\phi^{-1}(C)$ is then a compact subset of $V$ since $\phi$ is proper. Exhaust $V$ by compact subsets $K_n$; then there exists an $N \in \mathbb{N}$ so that $\phi^{-1}(C) \subseteq K_N$. We can find $T \subseteq V$, a closed neighborhood of infinity of $M \setminus p$, and, since $V \subsetneq \mathcal{R}_U\setminus \mathring{D}$, also of $\mathcal{R}_U$, so that the image under $\phi$ lies entirely in $\mathcal{R}_U \setminus \mathring{D^4}$ (see Figure \ref{InsideU}).
\begin{figure}[htb]
\centering
\begin{tikzpicture}
\node[anchor=south west,inner sep=0] at (0,0){\includegraphics[width=14cm]{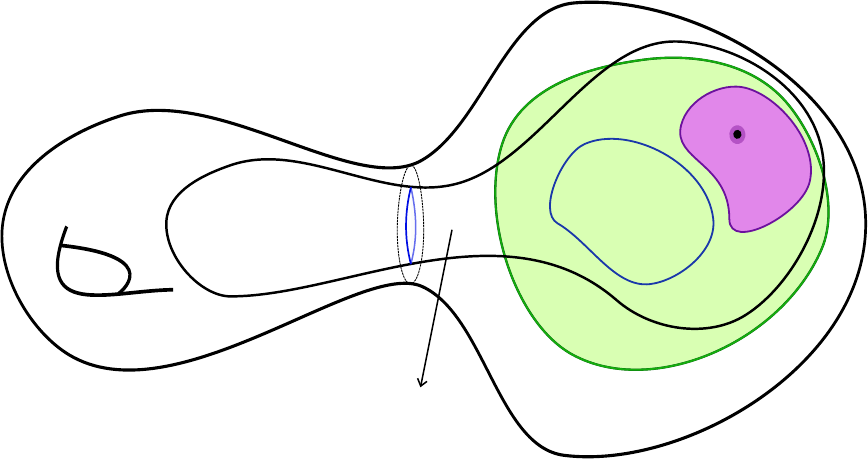}};

\node at (2.1,6) {$M \setminus \mathring{D'}$};
\node at (9.4,7.8) {$\mathcal{R}_U \setminus \mathring{D}$};
\node at (6.8,0.9) {$V'$};
\node at (12.3,4.7) {$\textcolor{violet}{T}$};
\node at (10.3,5.7) {$\textcolor{Green}{V}$}; 
\node at (7,3.8) {$\textcolor{Blue}{C}$};
\node at (10.3,4.5){$\textcolor{Blue}{\phi^{-1}(C)}$};
\end{tikzpicture}
\caption{Restricting $\phi$ to a diffeomorphism at infinity of $\mathcal{R}_U$. The intersection $C$ between $V' = \phi(V)$ and $\partial D$ is colored  in blue. The closed neighborhood of infinity $T$ whose image under $\phi$ is contained in $\mathcal{R}_U\setminus \mathring{D}$ is colored in violet.}.
\label{InsideU}
\end{figure}

 Here $\phi$ defines a diffeomorphism at infinity of $\mathcal{R}_U \setminus \mathring{D}$, hence also of $\mathcal{R}_U$. By Lemma \ref{Gompf2}, $\phi$ can be extended to a diffeomorphism of $\mathcal{R}_U$.  For ease of notation, call it $\phi$ as well.

 The map $\phi$ defines a smooth embedding of $D$, the disk used to perform the connected sum, inside $\mathcal{R}_U$. By the \emph{Cerf-Palais disk theorem} \cite{Cerf,Palais}, there is an ambient isotopy $H$ of $\mathcal{R}_U$ between the two embeddings  $\phi|_{D}: D \hookrightarrow \mathcal{R}_U$ and $\nu: D \hookrightarrow \mathcal{R}_U$, where $\nu$ denotes the inclusion. Moreover, $H$ can be assumed to be trivial on the complement of a compact subset $K$ of $\mathcal{R}_U$. 
 
  Then $H_1:\mathcal{R}_U \to \mathcal{R}_U$ is a self-diffeomorphism of $\mathcal{R}_U$ satisfying $H_1 \circ \phi|_{D} = \nu $ and $H_1|_{\mathcal{R}_U\setminus K}=id$. The composition $H_1 \circ \phi$ defines a self-diffeomorphism of $\mathcal{R}_U$ which extends the diffeomorphism at infinity $\phi$ and restricts to the identity on the disk $D$ used to perform the connected sum. Call $H_1 \circ \phi =: \Phi$.
 By construction, $\Phi|_{\partial D}$= id; hence we can extend $\Phi$ via the identity over $M \setminus \mathring{D'}$, obtaining a self-diffeomorphism of $M \setminus p $, call it $\Phi \# id $, that extends $\phi$, the diffeomorphism at infinity of $M \setminus p$ (see Figure \ref{Universalcase}).
 This concludes the proof.
\begin{figure}[htb]
\centering
\begin{tikzpicture}
\node[anchor=south west,inner sep=0] at (0,0){\includegraphics[width=14cm]{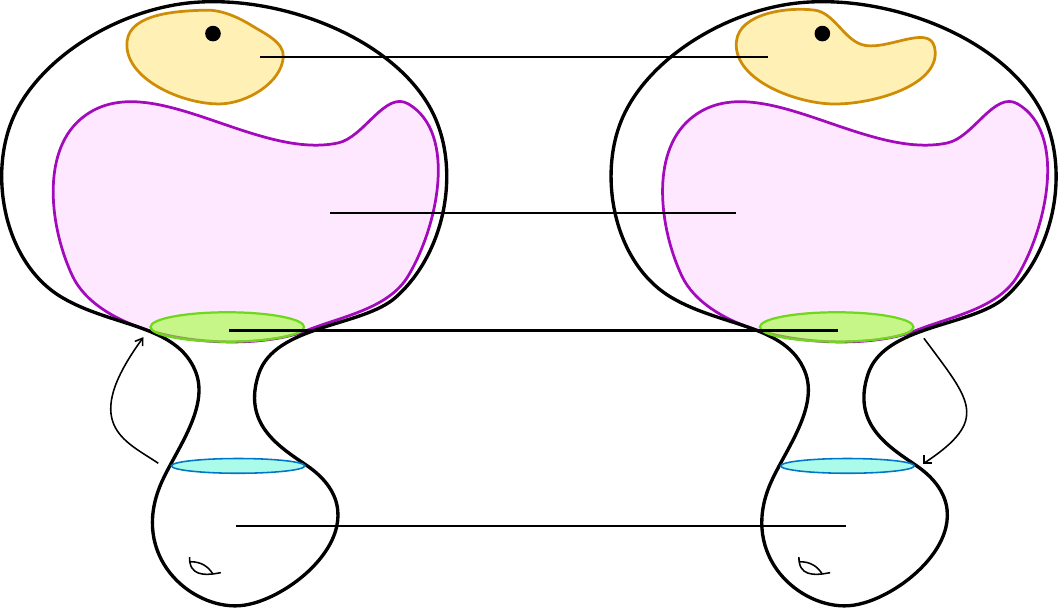}};
\node at (1.2,1.2) {$M \setminus \mathring{D'}$};
\node at (13.4,1.2) {$M \setminus \mathring{D'}$};
\node at (13.3,2.6){$\psi^{-1}$};
\node at (3,4.9) {$\textcolor{violet}{K}$};
\node at (11,4.9) {$\textcolor{violet}{K}$};
\node at (7.1,5.6) {$\textcolor{violet}{\Phi|_K}$};
\node at (7.1,1.4) {$id$};
\node at (7.1,4) {$\textcolor{Green}{\Phi|_{D}= id }$}; 
\node at (7.1,7.6) {$\textcolor{orange}{\Phi|_{T}=\phi}$};
\node at (1.2,2.6){$\psi$};
\node at (2.3, 7.3) {$\textcolor{orange}{T}$};
\node at (11, 7.2) {$\textcolor{orange}{\phi(T)}$};
\node at (2.8,8.4) {$\mathcal{R}_U\setminus \mathring{D}$};
\node at (11.1,8.4) {$\mathcal{R}_U\setminus \mathring{D}$};
\end{tikzpicture}
\caption{How the self-diffeomorphism $\Phi \# id$ of $M\#\mathcal{R}_U$ is constructed. The map $\Phi \# id $ embeds $T$ a closed neighborhood of infinity as $\phi(T)$. Inside the compact region $K$ we have $\Phi = H_1 \circ \phi$, with the property that $\Phi|_{D}=id$. Since the composition $\psi^{-1} \circ \Phi|_{\partial D} \circ \psi =id $, we can extend $\Phi$ via the identity over $M \setminus \mathring{D'}$.}
\label{Universalcase}
\end{figure}
\end{proof}
Now we can prove the general version of this result. As in the proof of Theorem \ref{new}, most of the techinques we will use to prove Theorem \ref{new1} come from the proof of Lemma \ref{theorem3}.

In the statement we use the term \emph{universal behavior}. This indicates that the smoothings of $M \setminus S$ satisfy a property similar to the one of $\mathcal{R}_U$ with respect to diffeotopy groups. Lemma \ref{theorem3} is analogous to Lemma \ref{Gompf2}. This is because both $M \setminus p$ and $\mathbb{R}^4$ have a single end. Under the assumptions of Lemma \ref{theorem3}, the diffeotopy group at infinity $\mathcal{D}^{\infty}(M \setminus p)$ coincides with the diffeotopy group $\mathcal{D}^{\epsilon_p}(M \setminus p)$ of the end at $p$. Moreover, the map $r$ coincides with $r_{\epsilon_p}$. When dealing with multiple punctures, and especially non-compact manifolds,  \emph{universal behavior} has a slightly different meaning which we now specify.

\begin{theorem}\label{new1}
Let $M$ be a topological 4-manifold, $S \subsetneq \mathring{M}$ a non-empty discrete set of points. 
\begin{enumerate}
    \item If $M$ is smoothable, then $M \setminus S$ admits a smoothing for which the map $r_{\epsilon_p}$ is surjective for each $p \in S$.
    \item If $M$ is non-smoothable and $|S| \ge 2$, then $M \setminus S$ admits a smoothing for which $r_{\epsilon_p}$ is surjective for all but one $p \in S$.
\end{enumerate}
\end{theorem}
In both cases we say that the smoothing has \emph{universal behaviour} (this does implicitly depend on the pair $(M,S)$, but we suppress this from the notation). Note that this does not impose any condition on other ends of $M \setminus S$.
\begin{proof}

We already covered case (1) when $|S|=1$ and $M$ is compact in Lemma \ref{theorem3}. To prove Theorem \ref{new1}, it suffices to prove case (1) for $M$ non-compact.  Case (2) follows from (1) applied to $M' := M \setminus p$ and $S' := S \setminus p$, where $p \in S$, again using Quinn's result \cite{endsofmaps}.

As in the proof of Theorem \ref{new}, to prove (1) it suffices to show the result is true for a non-compact 4-manifold $M$ and $|S|=1$. Once we have this, we will show how to obtain the general result. The compact case is analogous.
\begin{claim}
Let $M$ be a non-compact 4-manifold and $p \in \mathring{M}$. Then there is a smoothing of $M\setminus p$ with universal behavior.
\end{claim}

\begin{proof}[Proof of the claim]

The proof is analogous to that of Lemma \ref{theorem3}. Note that $\partial D'$ separates $\epsilon_p$ from all the other ends of the non-compact $M$. 
\end{proof}

Using the claim, we can prove (1). Suppose $S$ is discrete and let $\Lambda$ denote the index set. The set $\Lambda$ must be countable otherwise $S$ would not be discrete, so we can identify it with a subset of $\mathbb{Z}^+$. 

Pick $p_1 \in S$. Using the claim, we obtain a smoothing of $M \setminus p_1$ for which the map $r_{\epsilon_{p_1}}$ is surjective. Fix the smoothing and call $M'$ the resulting smooth manifold. Call $S':= S \setminus p_1$. Take $p_2 \in S$. Applying the claim to $M'$ and $p_2$, we get a smoothing of $M' \setminus p_2$ for which the map $r_{\epsilon_{p_2}}$ is surjective. Since $S$ is discrete, we can choose the disk $\mathring{D'}$ away from $p_1$ to perform the connected sum with $U$. Then the smooth structure is not changed near the puncture $p_1$. Hence $r_{\epsilon_{p_1}}$ remains surjective.

Repeat the procedure for each index $i \in \Lambda$. In the end we get a smoothing of $M \setminus S$ for which the maps $r_{\epsilon_{p_i}}$ are surjective for each $i \in \Lambda$. Thus, we have a smoothing of $M \setminus S$ with \emph{universal behavior}. 
As mentioned, we can prove the case $M$ compact, smoothable and $|S| \ge 2$ analogously.

To prove case (2), apply (1) to $M':= M \setminus p$ and $S':= S \setminus p$ using Quinn's \emph{smoothablity theorem} \cite{endsofmaps}. Keep in mind that this step causes loss of information about the end $\epsilon_p$. We obtain a smoothing of $M \setminus S$ for which the maps $r_{\epsilon_{p_i}}$ are surjective for $p_i \neq p$. However, we do not  know whether the same holds for the map $r_{\epsilon_p}$.
\end{proof}

\clearpage
\printbibliography[
heading=bibintoc,
title={References}
]

\end{document}